\documentclass[12pt]{amsart}
\usepackage{amsmath,amssymb,txfonts}

      \newtheorem{theorem}{Theorem}[section]

      \newcommand{\ct}[1]{\langle {#1}\rangle \lower.3ex\hbox{$_{t}$}}
      \newcommand{\lt}[1]{[ {#1}] \lower.3ex\hbox{$_{t}$}}

\newtheorem*{mthm}{Longinetti's convexity}

\begin{document}

\title[Geometric Capacity Potentials on Convex Plane Rings]{Geometric Capacity Potentials on Convex Plane Rings}

\author{Jie Xiao}
\address{Department of Mathematics and Statistics, Memorial University, St. John's, NL A1C 5S7, Canada}
\email{jxiao@mun.ca}
\thanks{This project was supported by NSERC of Canada and Memorial's University Research Professorship.}


\date{}



\begin{abstract}
Under $1<p\le 2$, this paper presents some old and new convexity/isoperimetry based inequalities for the variational $p$-capacity potentials on convex plane rings.
\end{abstract}
\maketitle


\section{Introduction}\label{s1}
\setcounter{equation}{0}

A pair $(K,\Omega)$ of sets in the plane $\mathbb R^2$ will be called a condenser if $\Omega$ is a open subset of $\mathbb R^2$ and $K$ is a compact subset of $\Omega$. And, a condenser $(K,\Omega)$ will be called a ring whenever $\Omega\setminus K$ is connected and $(\{\infty\}\cup\mathbb R^2)\setminus(\Omega\setminus K)$ comprises only two components (cf. \cite{Sar}). Under $1<p\le 2$ the variational $p$-capacity of a given condenser $(K,\Omega)$ is defined as 
$$
pcap(K,\Omega)=\inf\left\{\int_{\Omega}|\nabla f|^p\,dA:\ \ f\in \dot{W}^{1,p}(\Omega),\ \ f\ge 1_K\right\}
$$
where $\dot{W}^{1,p}(\Omega)$ is the completion of all infinitely differentiable functions $g$ on $\mathbb R^2$ with compact support in $\Omega$ under the Sobolev semi-norm $\big(\int_{\Omega}|\nabla g|^p\,dA\big)^\frac1p<\infty$, $dA$ is the differential element of area on $\mathbb R^2$, and $1_K$ represents the indicator function of $K$. A function that realizes the above infimum is called a $p$-capacity potential. When the infimum is finite, there exists a unique $p$-capacity potential $u$ which solves the following boundary value problem:

\begin{equation}
\label{e1a}
\begin{cases}
div(|\nabla u|^{p-2}\nabla u)\big|_{\Omega\setminus K}=0;\\
u|_{\partial\Omega}=0;\\
u|_{K}=1.
\end{cases}
\end{equation}

As an interesting topic in calculus of variations and mathematical physics (see e.g. \cite{Flu, Pol}), the variational $p$-capacity problem is to study the essential properties of a solution $u$ to (\ref{e1a}) and hence of the capacity $pcap(K,\Omega)$. If a given condenser $(K,\Omega)$ is good enough - for example - $(K,\Omega)$ is a ring with $K$ and $\Omega$ being convex, then an integration-by-parts can be performed to establish the level curve representation (cf. \cite{Rom, PP}) of the $p$-capacity of $(K,\Omega)$:
\begin{equation}\label{e1b}
pcap(K,\Omega)=\int_{\{z\in \overline{\Omega}\setminus K^\circ:\ \ u(z)=t\}}|\nabla u|^{p-1}\,dL\quad\forall\quad t\in [0,1],
\end{equation}
where $\overline{\Omega}$, $K^\circ$ and $dL$ stand for the closure of $\Omega$, the interior of $K$ and the differential element of length on $\mathbb R^2$ respectively, and 
$$
|\nabla u(x)|\equiv\liminf_{\Omega\setminus K\ni y\to x}|\nabla u(y)|=\limsup_{\Omega\setminus K\ni y\to x}|\nabla u(y)|\quad\forall\quad x\in \partial(\overline{\Omega}\setminus K^\circ).
$$
In accordance with \cite[page 13]{Flu}, the $p$-modulus of a condenser $(K,\Omega)$ is decided by
\begin{equation}
\label{e1c}
pmod(K,\Omega)=\big(pcap(K,\Omega)\big)^\frac{1}{1-p}.
\end{equation}
Of course, it is well-known that if $p=2$ then (\ref{e1a}), (\ref{e1b}) and (\ref{e1c}) are the conformal capacity problem, the conformal capacity and the conformal modulus of $(K,\Omega)$ respectively.
 
An important limiting case of (\ref{e1a}) should get special treatment. More precisely, if $K$ reduces to a point $o\in \Omega$, then we are led to consider the singular $p$-capacity potential -- the $p$-Green function $g_{\Omega}(o,\cdot)$ of $\Omega$ with singularity at $o$:

\begin{equation}
 \label{e1d}
 \begin{cases}
 div(|\nabla g_{\Omega}(o,\cdot)|^{p-2}\nabla g_{\Omega}(o,\cdot))\big|_{\Omega\setminus K}=-\delta(o,\cdot);\\
 g_{\Omega}(o,\cdot)|_{\partial\Omega}=0,
 \end{cases}
 \end{equation}
where $\delta(o,\cdot)$ is the Dirac distribution at $o$. According to \cite[pages 8, 13 and 63]{Flu}, if
\begin{equation}\label{e1e}
k_p(r)=\begin{cases}
\Big(\frac{p-1}{2-p}\Big)(2\pi)^\frac1{1-p}r^\frac{p-2}{p-1}\quad\hbox{for}\quad 1<p<2;\\
(2\pi)^{-1}\ln r^{-1} \quad\hbox{for}\quad p=2,
\end{cases}
\end{equation}
then
\begin{equation}\label{e1f}
\tau_p(o,\Omega)=\lim_{z\to o}\big(k_p(|z-o|)-g_{\Omega}(o,z)\big)
\end{equation}
is called the $p$-Robin function of $\Omega$ at $o$, and hence the function $\rho_p(o,\Omega)$ determined by 
\begin{equation}
\label{e1g}
k_p(\rho_p(o,\Omega))=\tau_p(o,\Omega)
\end{equation}
is said to be the $p$-harmonic (conformal as $p=2$) radius. Interestingly, such a radius can be utilized to estimate $pcap\big(\overline{D(o,r)},\Omega\big)$ for $\overline{D(o,r)}$ being the closed disk centered at $o$ with radius $r\to 0$ (cf. \cite[pages 81-82]{Flu}):
\begin{equation}\label{e1h}
\tau_p(o,\Omega)=\begin{cases}
\lim_{r\to 0}\frac{pcap\big(\overline{D(o,r)},\Omega\big)-pcap\big(\overline{D(o,r)},\mathbb R^2\big)}{(p-1)\big(pcap\big(\overline{D(o,r)},\mathbb R^2\big)\big)^\frac{p}{p-1}}\quad\hbox{for}\quad 1<p<2;\\
\lim_{r\to 0}\Big(\big(pcap\big(\overline{D(o,r)},\Omega\big)\big)^{-1}-(2\pi)^{-1}\ln r^{-1}\Big)\quad\hbox{for}\quad p=2,
\end{cases}
\end{equation}
where
\begin{equation}
\label{e1i}
pcap\big(\overline{D(o,r)},\mathbb R^2\big)=\begin{cases}
2\pi\Big(\frac{p-1}{2-p}\Big)^{1-p}r^{2-p}\quad\hbox{for}\quad 1<p<2;\\
r\quad\hbox{for}\quad p=2.
\end{cases}
\end{equation}

A careful look at (\ref{e1b}) reveals that the level curve of the $p$-capacity potential $u$ of (\ref{e1a}) plays a decisive role. Moreover, for $t\in [0,1]$, let $A(t)$ be the area of the set bounded by the closed level curve $\{z\in \overline{\Omega}\setminus K^\circ:\ u(z)=t\}$ whose length is denoted by $L(t)$. Then we have the isoperimetric inequality below:
\begin{equation}
\label{e1n}
A(t)\le (4\pi)^{-1}{\big(L(t)\big)^2}.
\end{equation}
This (\ref{e1n}), together with the closely related works \cite{Log, Ale, Lau}, suggests us to further find geometric properties induced by  (\ref{e1a})-(\ref{e1i}) such as optimal estimates for the area and perimeter of a level set of either $p$-capacity potential of a convex ring and or $p$-Green function of a bounded convex domain, as well as sharp isoperimetric inequalities involving $p$-capacity  - the details will be respectively presented in the forthcoming three sections: \S\ref{s2}-\S\ref{s3}-\S\ref{s4}.

\section{Longinetti's convexity for $p$-capacity potentials of convex rings}\label{s2}
\setcounter{equation}{0}

 Referring to \cite[Section 2]{Log}, in what follows we always suppose that $\Omega$ is a planar convex domain containing the origin, $\nu=(\cos\theta,\sin\theta)$ is the exterior unit normal vector to the boundary $\partial\Omega$ at the point $z=(z_1,z_2)\in\partial\Omega$, and
\begin{equation}
\label{e2a}
h(\theta)=z\cdot\nu=z_1\cos\theta+z_2\sin\theta
\end{equation}  
is the support function $h_K$ of $K=\overline{\Omega}$ which measures the Euclidean distance from the origin to the support line $\ell$ supporting $\partial\Omega$ at $z$ orthogonal to $\nu$. Here it is worth recalling the differentiability and integrability of the support function below: 

\begin{itemize}
\item if $\partial\Omega$ is strictly convex and $\ell$ supports $\partial\Omega$ at $z$ only then $h$ is of class $C^1$ and
  \begin{equation}
  \label{e2b}
  h'(\theta)=\frac{d}{d\theta}h(\theta)=-z_1\sin\theta+z_2\cos\theta;
  \end{equation}
\item if $\partial\Omega$ is of class $C^2$ and its curvature $\kappa(\theta)$ is positive then $h$ is of class $C^2$ and
\begin{equation}
\label{e2c}
h''(\theta)=\frac{d^2}{d\theta^2}h(\theta)=\big(\kappa(\theta)\big)^{-1}-h(\theta).
\end{equation}

\item the area $A=A(\Omega)$ of $\Omega$ and the length $L=L(\Omega)$ of $\partial\Omega$ are determined by
\begin{equation}
\label{e2d}
\begin{cases}
A=2^{-1}\int_{\mathbb S^1}h(\theta)\big(\kappa(\theta)\big)^{-1}\,d\theta;\\
L=\int_{\mathbb S^1}h(\theta)\,d\theta=\int_{\mathbb S^1}\big(\kappa(\theta)\big)^{-1}\,d\theta,
\end{cases}
\end{equation}
where $\mathbb S^1$ is the unit circle and may be identified with the interval $[0,2\pi)$.

\end{itemize}

With the help of (\ref{e2a})-(\ref{e2d}), Longinetti obtained the following assertion - see also \cite[Theorems 3.1-3.2]{Log}:

\begin{mthm}\label{t2a} Given $p\in (1,2]$ and two convex domains $\Omega_0$ and $\Omega_1$ with $0\in\Omega_1\subset \overline{\Omega}_0\subset\mathbb R^2$. For a solution $u$ to (\ref{e1a}) with $K=\overline{\Omega}_1$ and $\Omega=\Omega_0$ and $t\in [0,1]$ let $\Omega_t$ be the convex domain bounded by the level curve ${\Gamma_t}\equiv\{z\in{\overline{\Omega}_0}\setminus\Omega_1:\ u(z)=t\}$ as well as $A(t)$ and $L(t)$ be the area of $\Omega_t$ and the length of $\partial\Omega_t={\Gamma_t}$ respectively. Then
\begin{equation}
\label{e2e}
\begin{cases}
A'(t)A'''(t)-2p^{-1}\big(A''(t)\big)^2\ge 0;\\
L(t)L''(t)-(p-1)^{-1}\big(L'(t)\big)^2\ge 0,
\end{cases}
\end{equation}
where each equality in (\ref{e2e}) holds when and only when all level curves $\{{\Gamma_t}\}_{t\in [0,1]}$ are concentric circles.    
\end{mthm}
 
 In order to work out the area-analogue of the second inequality in (\ref{e2e}), we have the following assertion whose (i) and (ii) with $p=2$ are due to Longinetti - see also \cite[(3.28), (3.29) and (5.13)]{Log}.

\begin{theorem}\label{t2b} Under the same assumptions as in Longinetti's convexity, one has:
\item{\rm(i)} if $\partial\Omega_1$ is a circle then 
\begin{equation}
\label{e2h}
2(p-1)A(t)A''(t)\ge p\big(A'(t)\big)^2.
\end{equation}

\item{\rm(ii)} if $|\nabla u|$ equals a constant on $\partial\Omega_0$ then
 \begin{equation}
 \label{e2i}
 2(p-1)A(t)A''(t)\le p\big(A'(t)\big)^2.
 \end{equation}
\end{theorem} 
\begin{proof} (i) This follows from \cite[(3.29)]{Log}.

(ii) To verify (\ref{e2i}), let
$$
M_p(t)=\frac{A(t)A''(t)-p(2p-2)^{-1}\big(A'(t)\big)^2}{\big(A(t)\big)^\frac1{p-1}}\quad\forall\quad t\in [0,1].
$$
Then a straightforward computation, along with the first inequality of (\ref{e2e}) and the fact $A'(t)\le 0$, gives
\begin{align*}
M_p'(t)&=\big(A(t)\big)^\frac1{1-p}\Big(A'''(t)A(t)+A''(t)A'(t)-p(2p-2)^{-1}A'(t)A''(t)\\
&\ \ \ \ +\big(A''(t)A'(t)-p(2p-2)^{-1}(A'(t))^3(A(t))^{-1}\big)(1-p)^{-1}\Big)\\
&\le 2p^{-1}\big(A'(t)(A(t))^\frac{2-p}{1-p}\big(M_p(t)\big)^2\le 0.
\end{align*}
As a consequence, one has
$$
M_p(t)\le M_p(0)\quad\forall\quad t\in [0,1].
$$
If $M_p(0)\le 0$ then the argument is complete. To see this, we are required to calculate $A'(0)$ and $A''(0)$. Note that $|\nabla u|$ is a constant, say, $c$ on $\partial\Omega_0$. So, an application of (\ref{e1b}) yields
$$
pcap(K_1,\Omega_0)=c^{p-1}L(0).
$$
For each $t\in [0,1]$ suppose that $h(\theta,t)$ and $\kappa(\theta,t)$ are the support function and the curvature function of the level curve ${\Gamma_t}=\{z\in \overline{\Omega}_0\setminus\Omega_1: u(z)=t\}$. Then, according to the Lewis convexity in \cite[Theorem 1]{Lewis}, one has that for each $t\in (0,1)$ the level curve ${\Gamma_t}$ is strictly convex and $|\nabla u|\not=0$ in $\overline{\Omega}_0\setminus\Omega_1$. Moreover, one has the following formulas for the first-order and second-order derivatives of 
\begin{equation}
\label{e2j-}
\begin{cases} A(t)=2^{-1}\int_{\mathbb S^1}h(\theta,t)\big(\kappa(\theta,t)\big)^{-1}\,d\theta;\\
L(t)=\int_{\mathbb S^1}h(\theta,t)\,d\theta,
\end{cases}
\end{equation}
(cf. \cite[3.11-3.12 and 3.5-3.6]{Log}):
\begin{equation}
\label{e2j}
\begin{cases}
A'(t)=\int_{\mathbb S^1}\big(\frac{\partial}{\partial t}h(\theta,t)\big)\big(\kappa(\theta,t)\big)^{-1}\,d\theta;\\
A''(t)=\int_{\mathbb S^1}\Big(\big(\frac{\partial}{\partial t}h(\theta,t)\big)\frac{\partial}{\partial t}\big(\kappa(\theta,t)\big)^{-1}+\big(\frac{\partial^2}{\partial t^2}h(\theta,t)\big)\big(\kappa(\theta,t)\big)^{-1}\Big)\,d\theta;\\
 L'(t)=\int_{\mathbb S^1}\frac{\partial}{\partial t}h(\theta,t)\,d\theta;\\
 L''(t)=\int_{\mathbb S^1}\frac{\partial^2}{\partial t^2}h(\theta,t)\,d\theta.
\end{cases}  
\end{equation}
The preceding formulas (\ref{e2j-})-(\ref{e2j}), along with \cite[(2.13)]{Log} which particularly ensures 
\begin{equation}
\label{e2k}
|\nabla u|\big|_{\partial\Omega_0}\frac{\partial}{\partial t}h(\theta,t)\Big|_{t=0}=-1\\
\end{equation}
and so (cf. \cite[(3.20)]{Log})
$$
\frac{\partial^2}{\partial t^2}h(\theta,t)\big|_{t=0}=(p-1)^{-1}\left(\frac{\partial}{\partial t}h(\theta,t)\Big|_{t=0}\right)^2\kappa(\theta,t)\big|_{t=0},
$$
give
$$
A'(0)=-L(0)c^{-1}=-\big(L(0)\big)^\frac{p}{p-1}\big(pcap(\overline{\Omega}_1,\Omega_0)\big)^\frac1{1-p}
$$
and
$$
A''(0)=-L'(0)c^{-1}+2\pi(p-1)^{-1}c^{-2}=2\pi p(p-1)^{-1}\left(\frac{L(0)}{pcap(\overline{\Omega}_1,\Omega_0)}\right)^\frac2{p-1}.
$$
Now, using the case $t=0$ of (\ref{e1n}) one gets
$$
M_p(0)=\left(\frac{p(2p-2)^{-1}}{(pcap(\overline{\Omega}_1,\Omega_0))^\frac2{p-1}}\right)\left(\frac{(L(0))^2}{(A(0))^{2-p}}\right)^\frac1{p-1}\left(4\pi-\frac{(L(0))^2}{A(0)}\right)\le 0,
$$
as desired.
\end{proof} 
 
 \section{Isoperimetry for $p$-capacities of convex rings and bounded condensers}\label{s3} 
 \setcounter{equation}{0}
 
 In \cite[Theorem 1.1]{HS} (cf. \cite[Theorem 4.2]{Lewis1}) Henrot-Shahgholian showed that for $1<p\le 2$, a bounded convex domain $\Omega_1\subset\mathbb R^2$ and a constant $c>0$, there is a unique convex domain $\Omega_0\supset K_1=\overline{\Omega}_1$ in $\mathbb R^2$ such that
 \begin{equation}
 \label{e2l}
 \begin{cases}
 div(|\nabla u|^{p-2}\nabla u)\big|_{\Omega_0\setminus K_1}=0;\\
 u|_{\partial\Omega_0}=0;\\
 u|_{\partial\Omega_1}=1;\\
 |\nabla u|\big|_{\partial\Omega_0}=c.
 \end{cases}
 \end{equation}
 This fact leads to an isoperimetry for $pcap(\overline{\Omega}_1,\Omega_0)$ that extends Carleman's inequality  \cite[(5.10)]{Log} (cf. \cite[Proposition A.1]{SS}), Longinetti's inequality \cite[(5.4)]{Log} (cf. \cite[Theorem]{Fra} for another lower bound estimate for the case $p=2$) and Longinetti's isoperimetric deficit monotonicity \cite[(5.12)]{Log}.
 
 \begin{theorem}\label{t2c} For
 $1<p\le 2, c>0$ and $K_1=\overline{\Omega}_1$ let 
 $c_p=\big((2\pi)^{-1}pcap(K_1,\Omega_0)\big)^\frac1{p-1}$, $\Omega_0\supset K_1$
 and $(u,\Omega_0\setminus K_1,c)$ satisfy (\ref{e2l}).
 Then one has:
 \item{\rm(i)} an isoperimetry for the variational capacity 
 \begin{equation}
\label{e2m}
\begin{cases}
\big(\frac{L(1)}{2\pi}\big)^\frac{p-2}{p-1}-\big(\frac{L(0)}{2\pi}\big)^\frac{p-2}{p-1}\le{(\frac{2-p}{p-1})}{c_p}^{-1}\le \big(\frac{A(1)}{\pi}\big)^\frac{p-2}{2(p-1)}-\big(\frac{A(0)}{\pi}\big)^\frac{p-2}{2(p-1)}\ \ \hbox{for}\ \ 1<p<2;\\
\ln\Big(\frac{L(0)}{L(1)}\Big)\le{c_p}^{-1}\le\ln\Big(\frac{A(0)}{A(1)}\Big)^\frac12\quad\hbox{for}\quad p=2,
\end{cases}
\end{equation}
where (\ref{e2m}) holds with the sign of equality if $\Omega_0\setminus K_1$ is a circular annulus. 

\item{\rm(ii)} a monotonicity for the isoperimetric deficit 
\begin{equation}
\label{e2n}
\frac{d}{dt}\Big(\big(L(t)\big)^2-4\pi A(t)\Big)\ge 0\quad\forall\quad t\in [0,1],
\end{equation}
and consequently
\begin{equation}
\label{e2na}
\Big(\frac{L(0)}{2\pi}\Big)^2-\Big(\frac{L(1)}{2\pi}\Big)^2\le\Big(\frac{A(0)}{\pi}\Big)-\Big(\frac{A(1)}{\pi}\Big),
\end{equation}
with equality if $\Omega_0\setminus K_1$ is a circular annulus.
\end{theorem}
\begin{proof} (i) The case $p=2$ of (\ref{e2m}) can be seen from \cite[(5.4) and (5.10)]{Log}. So it remains to check the case $1<p<2$ of (\ref{e2m}). 

Let us begin with proving the left-hand-inequality of (\ref{e2m}) in
the case $1<p<2$. Note that the second inequality of (\ref{e2e}) is equivalent to
$$
\left(\frac{L''(t)}{-L'(t)}\right)\ge (p-1)^{-1}\left(\frac{-L'(t)}{L(t)}\right).
$$
So, integrating this inequality over $[0,t]\subset[0,1]$ derives
\begin{equation}
\label{e2o}
\frac{L'(t)}{L'(0)}\le \left(\frac{L(t)}{L(0)}\right)^{(p-1)^{-1}},
\end{equation}
namely,
$$
\big(L(t)\big)^{-(p-1)^{-1}}L'(t)\ge{L'(0)}{\big(L(0)\big)^{-(p-1)^{-1}}}.
$$
An integration of the last inequality over $[0,1]$ yields that if $1<p<2$ then
$$
\Big(\frac{p-1}{p-2}\Big)\Big(\big(L(1)\big)^\frac{p-2}{p-1} -\big(L(0)\big)^\frac{p-2}{p-1}\Big)\ge{L'(0)}{\big(L(0)\big)^{-(p-1)^{-1}}}.
$$
Since $|\nabla u|$ is just the constant $c$ on ${\partial\Omega_0}$, an application of (\ref{e1b}) and (\ref{e2j}), and (\ref{e2k}) gives
\begin{equation}
\label{e0e}
\frac{L'(0)}{\big(L(0)\big)^{(p-1)^{-1}}}=-\frac{(2\pi)^\frac{p-2}{p-1}}{c_p}.
\end{equation}
Upon putting this formula into the right-hand-side of the last inequality, we obtain the required estimate.

Next, let us verify the right-hand-inequality of (\ref{e2m}) in the case $1<p<2$.
In doing so, let $(D(0,r_0), D(0,r_1))$ be the origin-centered disk pair with
$A(j)=\pi r_j^2$ for $j=0,1$. Then an application of \cite[7.5 The main theorem]{Sar} and \cite[(2.13)]{HKM} gives
\begin{align*}
c_p&\ge\big((2\pi)^{-1}pcap(\overline{D(0,r_1)},D(0,r_0)\big)^\frac1{p-1}\\
&=\Big(\frac{2-p}{p-1}\Big)\Big(r_1^\frac{p-2}{p-1}-r_0^\frac{p-2}{p-1}\Big)^{-1}\\
&=\Big(\frac{2-p}{p-1}\Big)\left(\Big(\frac{A(1)}{\pi}\Big)^\frac{p-2}{2(p-1)}-\Big(\frac{A(0)}{\pi}\Big)^\frac{p-2}{2(p-1)}
\right)^{-1},
\end{align*}
as desired.

(ii) To establish (\ref{e2n}), we just observe (\ref{e1b}) and the following formula
$$
\begin{cases}
-A'(t)=\int_{{\Gamma_t}}|\nabla u|^{-1}\,dL;\\
\quad {2\pi}{c_p^{p-1}}=\int_{\Gamma_t}|\nabla u|^{p-1}\,dL,
\end{cases}
$$
and then utilize the H\"older inequality to achieve
$$
L(t)\le\Big(\int_{{\Gamma_t}}|\nabla u|^{p-1}\,dL\Big)^{p^{-1}}\Big(\int_{\Gamma_t}|\nabla u|^{-1}\,dL\Big)^{1-p^{-1}}=\big(2\pi c_p^{p-1}\big)^{p^{-1}}\big(-A'(t)\big)^{1-p^{-1}}.
$$
This, along with (\ref{e2o}) and (\ref{e0e}), derives
$$
\frac{d}{dt}\Big(\big(L(t)\big)^2-4\pi A(t)\Big)\ge-4\pi\left(\Big({2\pi}{c_p^{p-1}}\Big)^{-(p-1)^{-1}}\big(L(t)\big)^{p(p-1)^{-1}}+A'(t)\right)\ge 0,
$$
as desired. Of course, this gives (\ref{e2na}) right away.
\end{proof}

Needless to say, the right-hand-inequalities of (\ref{e2m}) are still valid for more general condensers. In addition, we can find out their pure capacity versions. Given a compact subset $K$ of $\mathbb R^2$. If $r>0$ is so large that $K$ is contained in the origin-centered open disk $D(0,r)$, then it is not hard to see that
  $$
  r\mapsto F_p(K,r)\equiv\begin{cases}
  \left(\Big(pcap\big(K,D(0,r)\big)\Big)^{\frac1{1-p}}+(2\pi)^{\frac1{1-p}}\Big(\frac{p-1}{2-p}\Big)r^{\frac{p-2}{p-1}}\right)^{1-p}\ \ \hbox{for}\ \ 1<p<2;\\
  \exp\left(-2\pi\Big(\Big(pcap\big(K,D(0,r)\big)\Big)^\frac{1}{1-p}+(2\pi)^{-1}\ln r^{-1}\Big)\right)\ \ \hbox{for}\ \ p=2,
  \end{cases}
  $$
  is a decreasing function on $[0,\infty)$ (cf. \cite[Lemma 2.1]{Flu} and \cite[Section 3.1]{Bet}). So, it is reasonable to define
  \begin{equation}
  \label{capD1}
  pcap(K)=\lim_{r\to\infty}F_p(K,r).
  \end{equation}
  Below is a known chain of the isocapacitary/isoperimetric inequalities (cf. \cite{Maz}, \cite[pages 140-141]{Ran} and \cite{AHN, HN, SoZ, BPS}) for the closure $\overline{\Omega}$ of a bounded domain $\Omega\subset\mathbb R^2$ with its area $A(\overline{\Omega})$ and diameter $diam(\overline{\Omega})$ as well as length $L(\overline{\Omega})$ of the boundary $\partial \overline{\Omega}$:
  \begin{equation*}
  \label{e4a}
  \Big(\frac{A(\overline{\Omega})}{\pi}\Big)^\frac12\le
  \begin{cases}
   \Big(\big(\frac{p-1}{2-p}\big)^{p-1}\big(\frac{pcap(\overline{\Omega}}{2\pi}\big)\Big)^\frac1{2-p}\ \ \hbox{for}\ \ 1<p<2\\
   pcap(\overline{\Omega})\ \ \hbox{for}\ \ p=2
   \end{cases}
   \le \frac{diam(\overline{\Omega})}{2}\le\frac{L(\overline{\Omega})}{4},
  \end{equation*}
  which holds with the sign of equality in the first two estimates if $\overline{\Omega}$ is a closed disk $\overline{D(0,r)}$ with (cf. (\ref{e1i}))
  $$
  pcap(\overline{D(0,r)})=pcap(\overline{D(0,r)},\mathbb R^2).
  $$
  
As an extension of \cite[Lemma 1]{Bet}, the following isocapacitary deficit result gives a sharp lower bound of $pcap(K,\Omega)$ in terms of $pcap(K)$ and $pcap(\overline{\Omega})$.

 \begin{theorem}\label{t2cc} Let $(K,\Omega)$ be a condenser in $\mathbb R^2$ with $\Omega$ being bounded.
 
 \begin{equation}
 \label{e2mm}
 \Big(\frac{pcap(K,\Omega)}{2\pi}\Big)^\frac{1}{1-p}\le
 \begin{cases}
  \big(\frac{pcap(K)}{2\pi}\big)^\frac{1}{1-p}-\big(\frac{pcap(\overline{\Omega})}{2\pi}\big)^\frac{1}{1-p}\ \ \hbox{for}\ \ 1<p<2;\\
 \ln\frac{pcap(\overline{\Omega})}{pcap(K)}\ \ \hbox{for}\ \ p=2,
 \end{cases}
 \end{equation}
 with equality if $\Omega\setminus K$ is a circular annulus. 
 \end{theorem}
 \begin{proof} 
 
For such a large $R>0$ that $\Omega\subset D(0,R)$, let
$$
{M(p,R)}\equiv\begin{cases}(2\pi)^\frac1{1-p}\Big(\frac{p-1}{2-p}\Big)R^\frac{p-2}{p-1}\ \ \hbox{for}\ \ 1<p<2;\\
(2\pi)^{-1}\ln R\ \ \hbox{for}\ \ p=2.
\end{cases}
$$
Then, an application of \cite[Lemma 2.1]{Flu} - the subadditivity of $p$-modulus and (\ref{capD1}) yields
\begin{align*}
 \big(pcap(K,\Omega)\big)^\frac1{1-p}&=pmod(K,\Omega)\\
 &\le pmod(K,D(0,R))+M(p,R)-pmod(\overline{\Omega},D(0,R))-M(p,R)\\
 &=\big(pcap(K,D(0,R))\big)^\frac1{1-p}+M(p,R)-\big(pcap(\overline{\Omega},D(0,R))\big)^\frac1{1-p}-M(p,R)\\
 &\to
 \begin{cases} \left(\big({pcap(K)}
 \big)^\frac{1}{1-p} - \big({pcap(\overline{\Omega})}\big)^\frac{1}{1-p}\right)\ \ \hbox{for}\ \ 1<p<2;\\
(2\pi)^{-1}\ln\frac{2cap(\overline{\Omega})}{2cap(K)}\ \ \hbox{for}\ \ p=2,
\end{cases} 
\quad \hbox{as}\quad R\to\infty,
 \end{align*}
 whence reaching (\ref{e2mm}) whose equality follows from the following formula for $0<r<R<\infty$:
  \begin{equation*}
  \label{e2e2}
  \left(\frac{pcap(\overline{D(0,r)}, D(0,R))}{2\pi}\right)^\frac1{1-p}=
  \begin{cases}\Big(\frac{p-1}{2-p}\Big)(R^\frac{p-2}{p-1}-r^\frac{p-2}{p-1})\ \hbox{for}\ \ 1<p<2;\\
  \ln\frac{R}{r}\ \ \hbox{for}\ \ p=2,
  \end{cases}
  \end{equation*}
  see also \cite[p.35]{HKM}. 
 
 \end{proof}
 \section{Convexity for $p$-Green functions of convex domains}\label{s4}
 \setcounter{equation}{0}
 
 The following is a generalization of \cite[Theorems 4.1-4.2]{Log} from $p=2$ to $p\in (1,2]$.
 
 \begin{theorem} \label{t31} Let $1<p\le 2$ and $\Omega\subset\mathbb R^2$ be a convex domain containing a given point $o$. For $t\ge 0$ set
 $$
 \begin{cases}
 A_g(t)=\int_{\{z\in\Omega: g_{\Omega}(o,z)\ge t\}}\,dA;\\
 L_g(t)=\int_{\{z\in\Omega: g_{\Omega}(o,z)=t\}}\,dL.
 \end{cases}
 $$
 Then
 
 \item{\rm(i)}
 
 \begin{equation}
 \label{e3a}
\begin{cases}
 A_g'(t)A_g'''(t)-2p^{-1}\big(A_g''(t)\big)^2\ge 0;\\
 L_g(t)L_g''(t)-(p-1)^{-1}\big(L_g'(t)\big)^2\ge 0,
\end{cases}
\end{equation} 
with equality if $\Omega$ is a disk centered at $o$.
 
 \item{\rm(ii)}
 \begin{equation}
  \label{e3b}
  \begin{cases}
  A_g(t)A_g''(t)\ge 2^{-1}p(p-1)^{-1}\big(A_g'(t)\big)^2;\\
  A_g''(t)\ge 2\pi p(p-1)^{-1}\big(A_g'(t)\big)^{2p^{-1}};\\
\big(A_g'(t)\big)^{2-\frac2p}\ge4\pi A_g(t),
\end{cases}
\end{equation}
with equality if $\Omega$ is a disk centered at $o$.

\item{\rm(iii)}
\begin{equation}
\label{e3c}
\begin{cases}
A_g(t)\le\begin{cases}
\Big(\big(A_g(0)\big)^\frac{p-2}{2p-2}+(\frac{2-p}{2p-2})(4\pi)^\frac{p}{2p-2}t\Big)^{\frac{2p-2}{p-2}}\ \ \hbox{for}\ \ 1<p<2;\\
A_g(0)\exp(-4\pi t)\ \ \hbox{for}\ \ p=2,
\end{cases}\\
L_g(t)\le\begin{cases}
\Big(\big(L_g(0)\big)^\frac{p-2}{p-1}+(\frac{2-p}{p-1})2\pi t\big)\Big)^\frac{p-1}{p-2}\ \ \hbox{for}\ \ 1<p<2;\\
L_g(0)\exp(-2\pi t)\ \ \hbox{for}\ \ p=2,
\end{cases}
\end{cases}
\end{equation}
with equality if $\Omega$ is a disk centered at $o$.
\end{theorem}
\begin{proof} (i) Since $\Omega$ is convex, each level curve of $g_\Omega(o,\cdot)$ is strictly convex (cf. \cite[Theorem 1]{Lewis}). This fact, plus (\ref{e2e}), implies (\ref{e3a}).

(ii) Next, noticing the following fundamental formula for $g_\Omega(o,\cdot)$ (cf. \cite[Lemma 9.1]{Flu}):
\begin{equation*}
\label{e3d}
\begin{cases}
-A_g'(t)=\int_{\{z\in\Omega: g_{\Omega}(o,z)=t\}}|\nabla g_{\Omega}(o,z)|^{-1} \,dL(z);\\
1=\int_{\{z\in\Omega: g_{\Omega}(o,z)=t\}}|\nabla g_{\Omega}(o,z)|^{p-1} \,dL(z),
\end{cases}
\end{equation*}
we employ the H\"older inequality to derive
\begin{equation}
\label{e3e}
L_g(t)\le\big(-A_g'(t)\big)^{1-p^{-1}}\left(\int_{\{z\in\Omega: g_{\Omega}(o,z)=t\}}|\nabla g_{\Omega}(o,z)|^{p-1} \,dL(z)\right)^{p^{-1}}=\big(-A_g'(t)\big)^{1-p^{-1}}.
\end{equation}
Clearly, the following isoperimetric inequality
\begin{equation}
\label{e3f}
A_g(t)\le (4\pi)^{-1}{\big(L_g(t)\big)^2}
\end{equation}
holds. So, a combination of (\ref{e3e}) and (\ref{e3f}) gives the third inequality of (\ref{e3b}). This, plus the first inequality of (\ref{e3b}), implies the second inequality of (\ref{e3b}). Thus, it remains to verify the first inequality of (\ref{e3b}). In doing so, let us choose a sequence of open disks $\{D_j\}_{j=1}^\infty$ centered at $o$ with radius tending to $0$. If
$$
\begin{cases}
a_j=\min_{z\in\partial D_j}g_{D_j}(o,z);\\
b_j=\max_{z\in\partial D_j}g_{D_j}(o,z),
\end{cases}
$$
then (\ref{e1f}) can be used to deduce $\lim_{j\to\infty}(b_j-a_j)=0$. Also, if $u_j$ and $v_j$ are $p$-harmonic in $\Omega\setminus D_j$, i.e.,
$$
div(|\nabla u_j|^{p-2}\nabla u_j)=0=div(|\nabla v_j|^{p-2}\nabla v_j)\quad\hbox{on}\quad\Omega\setminus D_j,
$$
subject to
$$
\begin{cases}
u_j(z)=v_j(z)=0\quad\forall\quad z\in\partial\Omega;\\
u_j(z)=b_j\quad\forall\quad z\in\partial D_j;\\
v_j(z)=a_j\quad\forall\quad z\in\partial D_j,\\
\end{cases}
$$
then an application of the comparison principle for $p$-harmonic functions (cf. \cite{HS}) derives
$$
v_j(z)\le g_{\Omega}(o,z)\le u_j(z)\quad\forall\quad z\in\Omega\setminus \overline{D_j}.
$$
This, together with $\lim_{j\to\infty}(b_j-a_j)=0$, implies that 
$$
\lim_{j\to\infty}u_j(z)=\lim_{j\to\infty}v_j(z)=g_{\Omega}(o,z)\quad\forall\quad z\in \Omega\setminus\overline{D(o,r)}
$$
holds for any small $r>0$ such that the open disk $D(o,r)$ is contained in $\Omega$. Now, using (\ref{e2h}) for $u_j$ and $v_j$ and letting $j\to\infty$ we arrive at the first inequality of (\ref{e3b}).

(iii) Finally, let us check (\ref{e3c}). Thanks to the first inequality of (\ref{e3b}) and the second inequality of (\ref{e3a}), it is enough to verify the area part of (\ref{e3c}). Note that the first inequality of (\ref{e3b}) yields that 
$$
t\mapsto A_g'(t)\big(A_g(t)\big)^{p(2-2p)^{-1}}
$$ 
is an increasing function on $[0,\infty)$. So, it follows that
$$
A_g'(t)\big(A_g(t)\big)^{p(2-2p)^{-1}}\le\lim_{s\to\infty}A_g'(s)\big(A_g(s)\big)^{p(2-2p)^{-1}}\equiv \gamma_p\quad\forall\quad t\in [0,\infty).
$$
This, along with an integration of the last inequality over $[0,t]$, gives
$$
A_g(t)\le\begin{cases}
\Big(\big(A_g(0)\big)^\frac{p-2}{2p-2}+(\frac{p-2}{2p-2})\gamma_p t\Big)^{\frac{2-2p}{p-2}}\ \ \hbox{for}\ \ 1<p<2;\\
A_g(0)\exp(\gamma_p t)\ \ \hbox{for}\ \ p=2.
\end{cases}
$$
Thus, it remains to show $\gamma_p=-(4\pi)^{p(2p-2)^{-1}}$. But, this follows from the basic fact (cf. \cite[Lemma 9.1]{Flu} and (\ref{e1g})) that when $t\to\infty$ the level set $\{z\in\Omega: g_{\Omega}(o,z)\ge t\}$ approaches a closed disk centered at $o$ with radius 
$$
r=\begin{cases}
\Big(\big({(2-p)}{(p-1)^{-1}}\big)(2\pi)^{(p-1)^{-1}}\Big)^{\frac{p-1}{p-2}}\big(t+\tau_p(o,\Omega)\big)^\frac{p-1}{p-2}\ \ \hbox{for}\ \ 1<p<2;\\
\exp\Big(-2\pi\big(t+\tau_p(o,\Omega)\big)\Big)\ \ \hbox{for}\ \ p=2.
\end{cases}
$$
\end{proof}

 \setcounter{equation}{0}

\end{document}